\RequirePackage{fix-cm}
\documentclass[smallextended]{svjour3}       
\smartqed  
\usepackage{graphicx}
\usepackage{subfigure}
\usepackage{amsmath,amssymb}
\usepackage{amsfonts}

\usepackage{listings,matlab-prettifier}
\usepackage[breaklinks,colorlinks,linkcolor=black,citecolor=black,urlcolor=black,hypertexnames=false]{hyperref}

\spnewtheorem{assumption}{Assumption}{\bf}{\rm}

\newcommand{\bR}{\mathbb{R}}

\newcommand{\cF}{\mathcal{F}}
\newcommand{\cH}{\mathcal{H}}

\newcommand{\rmd}{\mathrm{d}}

\newcommand{\bfb}{\mathbf{b}}
\newcommand{\bfc}{\mathbf{c}}
\newcommand{\bfe}{\mathbf{e}}
\newcommand{\bfu}{\mathbf{u}}
\newcommand{\bfv}{\mathbf{v}}
\newcommand{\bfy}{\mathbf{y}}

\newcommand{\bfA}{\mathbf{A}}
\newcommand{\bfF}{\mathbf{F}}
\newcommand{\bfL}{\mathbf{L}}

\newcommand{\Rn}{\mathbb{R}^{n}}
\newcommand{\Rnn}{\mathbb{R}^{n \times n}}

\newcommand{\zz}{^{\top}}

\newcommand{\uast}{^{\ast}}

\newcommand{\dkh}[1]{\left(#1\right)}
\newcommand{\hkh}[1]{\left\{#1\right\}}

\newcommand{\jkh}[1]{\left\langle#1\right\rangle}
\newcommand{\norm}[1]{\left\|#1\right\|}
\newcommand{\abs}[1]{\left\lvert #1\right\rvert}

\begin{document}

\title{A New Complexity Result for Strongly Convex Optimization with Locally $\alpha$-H{\"o}lder Continuous Gradients\thanks{We would like to acknowledge support for this project from RGC grant JLFS/P-501/24 for the CAS AMSS-PolyU Joint Laboratory in Applied Mathematics and Hong Kong Research Grant Council project PolyU15300024.}
}


\titlerunning{Complexity with H{\"o}lder Continuous Gradients}        

\author{Xiaojun Chen \and C. T. Kelley \and Lei Wang
	}

\authorrunning{X. Chen,  C. T. Kelley, and L. Wang} 

\institute{Xiaojun Chen \at
	Department of Applied Mathematics, The Hong Kong Polytechnic University, Hong Kong, China \\
	\email{maxjchen@polyu.edu.hk}           
	\and
	C. T. Kelley \at
	Department of Mathematics, Box 8205, North Carolina State University, Raleigh, NC 27695-8205, USA \\
	\email{Tim\_Kelley@ncsu.edu}
	\and
	Lei Wang \at
	Department of Applied Mathematics, The Hong Kong Polytechnic University, Hong Kong, China \\
	\email{wlkings@lsec.cc.ac.cn}
}

\date{Received: date / Accepted: date}

\maketitle

\begin{abstract}
	In this paper, we present a new complexity result for the gradient descent method with an appropriately fixed stepsize for minimizing a strongly convex function with locally $\alpha$-H{\"o}lder continuous gradients ($0 < \alpha \leq 1$). The complexity bound for finding an approximate minimizer with a distance to the true minimizer less than $\varepsilon$ is $O(\log (\varepsilon^{-1}) \varepsilon^{2 \alpha - 2})$, which extends the well-known complexity result for $\alpha = 1$.
\end{abstract}

\section{Introduction}
\label{sec:introduction}

Let $f: \Rn \to \bR$ be a $\mu$-strongly convex and continuously differentiable function with a parameter $\mu > 0$, that is,
\begin{equation*} 
	f (\bfu) - f (\bfv) \geq \nabla f (\bfv)\zz (\bfu - \bfv) + \frac{\mu}{2} \norm{\bfu - \bfv}^2,
\end{equation*}
for all $\bfu, \bfv \in \Rn$.
Here, $\norm{\,\cdot\,}$ is the $\ell_2$ norm on $\Rn$.
In this paper, we focus on establishing a new complexity result for the classic gradient descent method
\begin{equation} \label{eq:gd}
	\bfu_{k + 1} = \bfu_k - \tau \nabla f (\bfu_k),
\end{equation}
where $\tau > 0$ is a fixed stepsize, for solving the following unconstrained optimization problem
\begin{equation} \label{opt:main}
	\min_{\bfu \in \Rn} f (\bfu).
\end{equation}

Throughout this paper, we assume that the gradient $\nabla f$ of $f$ satisfies the locally $\alpha$-H{\"o}lder continuous condition.

\begin{definition} \label{def:Holder}
	We say $g: \Rn \to \Rn$ is \emph{locally $\alpha$-H{\"o}lder continuous} with $\alpha \in (0,1]$ if there exist $\beta > 0$ and $\delta > 0$ such that
	\begin{equation} \label{eq:Holder}
		\norm{g (\bfu) - g (\bfv)}
		\leq \beta \norm{\bfu - \bfv}^\alpha,
	\end{equation}
	for all $\bfu, \bfv \in \Rn$ satisfying $\norm{\bfu - \bfv} \leq \delta$.
	We denote this class of functions by $\cH (\alpha, \beta, \delta)$.
    We say $g: \Rn \to \Rn$ is \emph{$\alpha$-H{\"o}lder continuous} if condition \eqref{eq:Holder} holds for all $\bfu, \bfv \in \Rn$.
\end{definition}

Suppose that the gradient $\nabla f$ is globally Lipschitz continuous, which corresponds to condition \eqref{eq:Holder} with $\alpha = 1$ for all $\bfu, \bfv \in \Rn$.
Then it is well known that the gradient descent method \eqref{eq:gd} with any initial point $\bfu_0 \in \Rn$ and  stepsize $\tau=2/(\mu\beta)$ achieves a linear rate of convergence \cite[Theorem 2.1.15]{Nesterov2018lectures} as follows,
\begin{equation*}
    \norm{\bfu_k - \bfu\uast} \leq \dkh{\dfrac{\beta - \mu }{\beta + \mu}}^k \norm{\bfu_0 - \bfu\uast},
\end{equation*}
where $\bfu\uast$ is the minimizer of problem \eqref{opt:main}.
Therefore, for a given $\varepsilon > 0$, the gradient descent method \eqref{eq:gd} with stepsize $\tau=2/(\mu\beta)$ is guaranteed to find a point $\bfu_k$ satisfying $\norm{\bfu_k - \bfu\uast} \leq \varepsilon$ after at most $O (\log (\varepsilon^{-1}))$ iterations.
This analysis fails in the non-Lipschitz case.
As our theoretical results unveil, the gradient descent method \eqref{eq:gd} with a fixed stepsize may stagnate before converging to the minimizer.

The minimization of convex functions with $\alpha$-H{\"o}lder continuous gradients is first studied by Devolder et al. \cite{Devolder2014first} in the form of inexact oracles.
Following this work, Lan \cite{Lan2015bundle} develops an accelerated prox-level method for solving the class of composite problems.
Moreover, Nesterov \cite{Nesterov2015universal} proposes universal gradient methods for minimizing composite functions which have $\alpha$-H{\"o}lder continuous gradients of its smooth part.
However, to the best of our knowledge, the complexity bound of gradient descent method \eqref{eq:gd} for a function $f$ satisfying the locally $\alpha$-H{\"o}lder continuous condition with $0 < \alpha < 1$ has not been developed.

The contribution of this paper is to provide a new complexity result for the gradient descent method \eqref{eq:gd} with a fixed stepsize $\tau$ to solve strongly convex optimization problems on $\Rn$ when the gradient is locally $\alpha$-H{\"o}lder continuous with $0 < \alpha < 1$, but not Lipschitz continuous. We show that method \eqref{eq:gd} with an appropriately chosen $\tau$ can  find a point $\bfu_k$ satisfying $\norm{\bfu_k - \bfu\uast} \leq \varepsilon$ after at most $O (\log (\varepsilon^{-1})\varepsilon^{2\alpha-2})$ iterations.

We end this section by illustrating Definition \ref{def:Holder} with a simple example.

\begin{example}
	Consider the following univariate function,
	\begin{equation*}
		f (u) = \frac{\lambda}{2} u^2 + \frac{2}{3} u_+^{3/2},
	\end{equation*}
	where $\lambda > 0$, $u \in \bR$, and $u_+ = \max\{u, 0\}$.
	Then we have
	\begin{equation*}
		f^\prime (u) = \lambda u + \sqrt{u_+}.
	\end{equation*}
	It can be readily verified that $f^\prime$ cannot be expected to satisfy the condition \eqref{eq:Holder} globally for all $u, v \in \bR$.
	However, if $\norm{u - v} \leq \delta = 1$, it holds that
	\begin{equation*}
		\norm{u - v} \leq \norm{u - v}^{1/2}.
	\end{equation*}
	Moreover, for any $u, v \in \bR$, we have
	\begin{equation*}
		\norm{\sqrt{u_+} - \sqrt{v_+}}
		\leq \norm{u_+ - v_+}^{1/2}
		\leq \norm{u - v}^{1/2}.
	\end{equation*}
	Therefore, $f^\prime (u)$ satisfies the H{\"o}lder continuous condition \eqref{eq:Holder} with $\beta = 1 + \lambda$ and $\alpha = 1/2$ for $\norm{u - v} \leq \delta=1$.
\end{example}

\section{Convergence of Gradient Flow Equations}


We interpret the gradient descent iteration \eqref{eq:gd} as a forward Euler discretization of the following gradient flow equation,
\begin{equation} \label{eq:gflow}
\left\{
\begin{aligned}
	& \frac{\rmd \bfu}{\rmd t} = - \nabla f (\bfu), \\
	& \bfu (0) = \bfu_0,
\end{aligned}
\right.
\end{equation}
with time step $\tau > 0$.



The purpose of this section is to establish the existence and convergence of the unique solution to the initial value problem \eqref{eq:gflow}.
We state these results as Lemma \ref{le:solution} below.

\begin{lemma} \label{le:solution}
	Suppose that $f: \Rn \to \bR$ is $\mu$-strongly convex and continuously differentiable.
	Let $\bfu\uast \in \Rn$ be the unique minimizer of $f$.
	Then the initial value problem \eqref{eq:gflow} has a unique solution $\bfu (t)$ defined for all $t \in [0, \infty)$.
	Moreover, the trajectory $\{\bfu (t)\}_{t > 0}$ is bounded and
	\begin{equation*}
		\norm{\bfu (t) - \bfu\uast}
		\leq e^{- \mu t} \norm{\bfu_0 - \bfu\uast}.
	\end{equation*}
\end{lemma}

\begin{proof}
	The results in \cite[Theorem 2.19]{Teschl2012ordinary} indicate that a solution $\bfu (t)$ to problem \eqref{eq:gflow} exists.
	To prove the uniqueness, we assume that both $\bfu (t)$ and $\bfv (t)$ are solutions of problem \eqref{eq:gflow}.
	Then it follows from the strong convexity of $f$ that
	\begin{equation*}
		\frac{\rmd}{\rmd t} \norm{\bfu (t) - \bfv (t)}^2
		= -2 \dkh{\nabla f (\bfu(t)) - \nabla f (\bfv(t))}\zz \dkh{\bfu (t) - \bfv (t)}
		\leq 0,
	\end{equation*}
	which indicates that $\phi (t) = \norm{\bfu (t) - \bfv (t)}^2$ is decreasing in $t$.
	Since $\phi (0) = 0$, we know that $\phi (t) = 0$ for all $t \geq 0$.
	Hence, $\bfu (t) = \bfv (t)$.
	
	Similarly, according to the strong convexity of $f$, we have
	\begin{equation*}
	\begin{aligned}
		\frac{\rmd}{\rmd t} \norm{\bfu (t) - \bfu\uast}^2
		= {} & - 2 \nabla f (\bfu (t))\zz \dkh{\bfu (t) - \bfu\uast} \\
		= {} & - 2 \dkh{\nabla f (\bfu (t)) - \nabla f (\bfu\uast)}\zz \dkh{\bfu (t) - \bfu\uast} \\
		\leq {} & - 2 \mu \norm{\bfu (t) - \bfu\uast}^2.
	\end{aligned}
	\end{equation*}
	The Gr{\"o}nwall inequality \cite{Gronwall1919note} then implies that
	\begin{equation*}
		\norm{\bfu (t) - \bfu\uast}^2
		\leq e^{- 2 \mu t} \norm{\bfu_0 - \bfu\uast}^2,
	\end{equation*}
	as desired.
	\qed
\end{proof}

%
%
%
%
%
%
%


\section{Discretization Error}


This section delves into the error analysis for the forward Euler discretization of the gradient flow equation \eqref{eq:gflow}.
We first formalize the assumptions  on the function $f$ throughout the analysis.

\begin{assumption} \label{asp:function}
	The following statements hold.
	\begin{enumerate}
		
		\item $f$ is $\mu$-strongly convex with $\mu > 0$ and continuously differentiable.
		
		\item $\nabla f \in \cH (\alpha, \beta, \delta)$ for some $\beta, \delta > 0$ and $\alpha \in (0, 1]$.
		
	\end{enumerate}
\end{assumption}

According to Lemma \ref{le:solution}, we know that the initial value problem \eqref{eq:gflow} has a unique solution $\bfu (t)$.
Lemma \ref{le:Holder} summarizes some useful properties of $\bfu (t)$.

\begin{lemma} \label{le:Holder}
	Let Assumption~\ref{asp:function} hold.
	\begin{enumerate}
		
		
		\item There exists a constant $M > 0$ such that $\norm{\nabla f (\bfu (t))} \leq M$ for all $t \geq 0$.
		
		\item $\bfu$ is Lipschitz continuous with the Lipschitz constant $M$.
		
		\item There exists a constant $\bar{\beta} = \beta M^\alpha$ such that, for all $\tau \in (0, \delta / M)$, we have
		\begin{equation*}
			\norm{\bfu (t) - \tau \nabla f (\bfu (t)) - \bfu (t + \tau)} \leq \bar{\beta} \tau^{1 + \alpha}.
		\end{equation*}
		
	\end{enumerate}
\end{lemma}

\begin{proof}
	1. Since $\bfu (t)$ is finite for all $t \geq 0$ and $\bfu (t) \to \bfu\uast$ as $t \to \infty$, the trajectory $\{\bfu (t)\}_{t \geq 0}$ is bounded.
	Then the first assertion directly follows from the continuity of $\nabla f$.
	
	2. For all $t, w \geq 0$, it holds that
	\begin{equation*}
		\norm{\bfu (t) - \bfu (w)}
		= \norm{\int_w^t \nabla f (\bfu (s)) \rmd s}
		\leq \abs{\int_w^t \norm{\nabla f (\bfu (s))} \rmd s}
		\leq M \abs{t - w},
	\end{equation*}
	where the last inequality follows from the fact that $\norm{\nabla f (\bfu (t))} \leq M$ for all $t \geq 0$.
	Thus, the second assertion holds.
	
	3. Now we assume that $\tau \in (0, \delta / M)$.
	Then
	\begin{equation*}
		\norm{\bfu (t + \tau) - \bfu (t)}
		\leq M \tau < \delta,
	\end{equation*}
	which together with the local $\alpha$-H{\"o}lder continuity of $\nabla f$ implies that
	\begin{equation*}
		\norm{\nabla f (\bfu (t + \tau)) - \nabla f (\bfu (t))}
		\leq \beta \norm{\bfu (t + \tau) - \bfu (t)}^\alpha
		\leq \beta M^\alpha \tau^\alpha.
	\end{equation*}
	According to the differential equation in \eqref{eq:gflow}, it follows that
	\begin{equation*}
		\bfu (t) - \bfu (t + \tau) = \int_t^{t + \tau} \nabla f (\bfu (s)) \rmd s,
	\end{equation*}
	and hence,
	\begin{equation*}
		\begin{aligned}
			\norm{\bfu (t) - \tau \nabla f(\bfu (t)) - \bfu (t + \tau)}
			= {} & \norm{\int_t^{t + \tau} \dkh{\nabla f (\bfu (s)) - \nabla f (\bfu (t))} \rmd s} \\
			\leq {} &\int_t^{t + \tau}  \norm{\nabla f (\bfu (s)) - \nabla f (\bfu (t))} \rmd s \\
			\leq {} & \beta M^\alpha \tau^{1 + \alpha}.
		\end{aligned}
	\end{equation*}
	Hence, the third assertion holds.
	We complete the proof.
	\qed
\end{proof}

We are in the position to establish the discretization error of the initial value problem \eqref{eq:gflow}.

\begin{theorem}
	\label{thm:convergence}
	Suppose that Assumption \ref{asp:function} holds.
	Let $\{\bfu_k\}_{k = 0}^K$ be the forward Euler discretization to the solution $\bfu (t)$ of the initial value problem \eqref{eq:gflow} on $[0, T]$ with $K$ time steps and $\tau = T / K$.
	We further assume that
	\begin{equation} \label{eq:cond-K}
		K \geq \max \hkh{ T, \, 4 \mu T, \dfrac{M T}{\delta}, \, \dfrac{C_E^{1 / \alpha} T}{\delta^{1 / \alpha}}, \, \dfrac{(2 \beta^2)^{\frac{1}{(1 - \alpha)^2 + \alpha^2}} T^{1 + \frac{1}{(1 - \alpha)^2 + \alpha^2}}}{C_E^{\frac{2 (1 - \alpha)}{(1 - \alpha)^2 + \alpha^2}}} },
	\end{equation}
	where $C_E = \beta M^\alpha / \mu > 0$.
	Then, for any $0 \leq l \leq K$,
	\begin{equation} \label{eq:El}
		E_l \equiv \norm{\bfu_l - \bfu (\tau l)} \leq C_E \tau^\alpha.
	\end{equation}
\end{theorem}

\begin{proof}
	We seek to prove this theorem by induction.
	The inequality \eqref{eq:El} holds for $l = 0$ since $\bfu_0 = \bfu (0)$.
	Now we assume that \eqref{eq:El} holds for all $0 \leq l \leq k$ and investigate the situation for $l = k + 1$.
	
	For convenience, we define
	\begin{equation*}
		\bfy (t) = \bfu (t) - \tau \nabla f (\bfu (t)).
	\end{equation*}
	Let $t_k = \tau k$ and $\bfy_k = \bfy (t_k)$.
	Then a straightforward verification reveals that
	\begin{equation} \label{eq:yk-uk+1}
	\begin{aligned}
		\norm{\bfy_k - \bfu_{k + 1}}^2
		= {} & \norm{\bfu (t_k) - \bfu_k - \tau \dkh{\nabla f (\bfu (t_k)) - \nabla f (\bfu_k)}}^2 \\
		= {} & E_k^2
		- 2 \tau \dkh{\nabla f (\bfu (t_k)) - \nabla f (\bfu_k)}\zz \dkh{\bfu (t_k) - \bfu_k} \\
		& + \tau^2 \norm{\nabla f (\bfu (t_k)) - \nabla f (\bfu_k)}^2.
	\end{aligned}
	\end{equation}
	Moreover, the strong convexity of $f$ yields that
	\begin{equation} \label{eq:nabla-g}
		\dkh{\nabla f (\bfu (t_k)) - \nabla f (\bfu_k)}\zz \dkh{\bfu (t_k) - \bfu_k}
		\geq \mu \norm{\bfu (t_k) - \bfu_k}^2
		= \mu E_k^2.
	\end{equation}
	Since $\norm{\bfu (t_k) - \bfu_k} = E_k \leq C_E \tau^\alpha \leq \delta$, it follows from the local $\alpha$-H{\"o}lder continuity of $\nabla f$ that
	\begin{equation} \label{eq:nabla-l}
	\begin{aligned}
		\norm{\nabla f (\bfu (t_k)) - \nabla f (\bfu_k)}^2
		\leq {} & \beta^2 \norm{\bfu (t_k) - \bfu_k}^{2 \alpha} \\
		= {} & \beta^2 E_k^{2 \alpha}
		\leq \beta^2 C_E^{2 \alpha} \tau^{2 \alpha^2}.
	\end{aligned}
	\end{equation}
	Combining the three relationships \eqref{eq:yk-uk+1}, \eqref{eq:nabla-g}, and \eqref{eq:nabla-l} together, we can obtain that
	\begin{equation} \label{eq:u-y-l}
		\norm{\bfy_k - \bfu_{k + 1}}^2
		\leq \dkh{1 - 2 \mu \tau} E_k^2 + \beta^2 C_E^{2 \alpha} \tau^{2 + 2 \alpha^2}.
	\end{equation}
	Since $\tau \leq \delta / M$, it can be deduced from Lemma \ref{le:Holder} that
	\begin{equation*}
		\norm{\bfy_k - \bfu (t_{k + 1})} \leq \bar{\beta} \tau^{1 + \alpha}.
	\end{equation*}
	Then by virtue of the triangle inequality, we have
	\begin{equation*}
	\begin{aligned}
		\norm{\bfu_{k + 1} - \bfy_k}^2
		\geq {} & \abs{\norm{\bfu_{k + 1} - \bfu (t_{k + 1})} - \norm{\bfu (t_{k + 1}) - \bfy_k}}^2 \\
		= {} & E_{k + 1}^2 - 2 E_{k + 1} \norm{\bfu (t_{k + 1}) - \bfy_k} + \norm{\bfu (t_{k + 1}) - \bfy_k}^2 \\
		\geq {} & E_{k + 1}^2 - 2 E_{k + 1} \norm{\bfu (t_{k + 1}) - \bfy_k} \\
		\geq {} & E_{k + 1}^2 - 2 \bar{\beta} \tau^{1 + \alpha} E_{k + 1}.
	\end{aligned}
	\end{equation*}
	
	Now we assume for contradiction that $E_{k+1} > C_E \tau^\alpha$.
	Then it can be readily verified that
	\begin{equation} \label{eq:u-y-g}
	\begin{aligned}
		\norm{\bfu_{k + 1} - \bfy_k}^2
		\geq \dkh{1 - \dfrac{2 \bar{\beta} \tau}{C_E}} E_{k + 1}^2
		= \dkh{1 - 2 \mu \tau} E_{k + 1}^2.
	\end{aligned}
	\end{equation}
	In light of the relationship $\tau \leq 1 / (4 \mu)$, it can be obtained by combining \eqref{eq:u-y-l} and \eqref{eq:u-y-g} together that
	\begin{equation*}
		E_{k + 1}^2
		\leq E_k^2 + \dfrac{\beta^2 C_E^{2 \alpha} \tau^{2 + 2 \alpha^2}}{1 - 2 \mu \tau}
		\leq E_k^2 + 2 \beta^2 C_E^{2 \alpha} \tau^{2 + 2 \alpha^2}.
	\end{equation*}
	Since $E_0 = 0$, we can proceed to show that
	\begin{equation*}
	\begin{aligned}
		E_{k + 1}^2
		\leq {} & 2 (k + 1) \beta^2 C_E^{2 \alpha} \tau^{2 + 2 \alpha^2} \\
		\leq {} & 2 T \beta^2 C_E^{2 \alpha} \tau^{1 + 2 \alpha^2} \\
		= {} & 2 T \beta^2 C_E^{2 (\alpha - 1)} \tau^{(1 - \alpha)^2 + \alpha^2} \times C_E^{2} \tau^{2 \alpha} \\
		\leq {} & C_E^{2} \tau^{2 \alpha},
	\end{aligned}
	\end{equation*}
	where the last inequality follows from $\tau=T/K$ and condition \eqref{eq:cond-K}.
	This results in a contradiction to $E_{k+1} > C_E \tau^\alpha$.
	The proof is completed.
	\qed
\end{proof}

\section{A Basic Complexity Result}
\label{sec:complexity}

In this section, we apply Theorem \ref{thm:convergence} to obtain a complexity result of the gradient descent method \eqref{eq:gd} without any additional conditions.
We begin by defining the time required to get close to the minimizer of $f$.

\begin{definition}
	Let $\bfu (t)$ be the unique solution of problem \eqref{eq:gflow} and let $\bfu\uast = \lim_{t \to \infty} \bfu (t)$.
	For $\eta > 0$, we define
	\begin{equation*}
		T\uast (\eta) = \min \hkh{ T \,\middle|\, \norm{\bfu (t) - \bfu\uast} \leq \eta \mbox{ for all } t \geq T }.
	\end{equation*}
\end{definition}

According to Lemma \ref{le:solution}, we know that
\begin{equation*}
	\norm{\bfu (t) - \bfu\uast} \leq e^{-\mu t} \norm{\bfu_0 - \bfu\uast}.
\end{equation*}
Solving $e^{-\mu t} \norm{\bfu_0 - \bfu\uast} \leq \eta$ gives rise to that
\begin{equation*}
	T\uast (\eta) \leq \frac{1}{\mu} \log \dkh{ \frac{\norm{\bfu_0 - \bfu\uast}}{\eta} },
\end{equation*}
which provides an upper bound of $T\uast (\eta)$.

Next, we provide an estimate for the total number of time steps required by the forward Euler discretization to get close to the minimizer of $f$.

\begin{theorem}
	\label{thm:complexity}
	Suppose that Assumption \ref{asp:function} holds and $\eta > 0$ is a sufficiently small constant.
	Let $\{\bfu_k\}_{k = 0}^K$ be the forward Euler discretization to the solution $\bfu (t)$ of the initial value problem \eqref{eq:gflow} on $[0, T\uast (\eta)]$ with $K$ time steps and $\tau = T\uast (\eta) / K$.
	We further assume that $K$ satisfies condition \eqref{eq:cond-K} and
	\begin{equation*}
		K \geq K_{\bfu} (\eta) \equiv \mu^{-1/\alpha} \beta^{1/\alpha} M T\uast (\eta) \eta^{-1/\alpha}.
	\end{equation*}
	Then
	\begin{equation*}
		\norm{\bfu_K - \bfu (T\uast (\eta))} \leq \eta.
	\end{equation*}
\end{theorem}

\begin{proof}
	As a direct consequence of Theorem \ref{thm:convergence}, we can proceed to show that
	\begin{equation*}
		\norm{\bfu_K - \bfu (T\uast (\eta))}
		\leq C_E \tau^\alpha
		\leq \dfrac{\bar{\beta}}{\mu} \dkh{\frac{T\uast (\eta)}{K}}^\alpha
		\leq \eta,
	\end{equation*}
	where the last inequality is satisfied with $\bar{\beta} = \beta M^\alpha$ when $K \geq K_{\bfu} (\eta)$.
	The proof is completed.
	\qed
\end{proof}

It is noteworthy that the sequence $\{\bfu_{k}\}$ generated by the gradient descent method \eqref{eq:gd} with the initial point $\bfu_{0}$ is exactly the forward Euler discretization to the solution $\bfu (t)$ of the initial value problem \eqref{eq:gflow}.
Therefore, according to Theorem \ref{thm:complexity}, we can derive a complexity result of the gradient descent method \eqref{eq:gd} with a fixed stepsize, which is demonstrated in the following corollary.


\begin{corollary}
	\label{coro:complexity}
	Suppose that Assumption \ref{asp:function} holds and $\varepsilon > 0$ is a sufficiently small constant.
	Let $\{\bfu_k\}$ be the sequence generated by the gradient descent method \eqref{eq:gd} with the stepsize $\tau$ chosen as
	\begin{equation*}
		\tau = 2^{-1/\alpha} \mu^{1/\alpha} \beta^{-1/\alpha} M^{-1} \varepsilon^{1/\alpha}.
	\end{equation*}
	Then after at most
	\begin{equation*}
		\left\lceil 2^{1/\alpha} \mu^{-1-1/\alpha} \beta^{1/\alpha} M \log \dkh{\dfrac{2 \norm{\bfu_0 - \bfu\uast}}{\varepsilon}} \varepsilon^{-1/\alpha} \right\rceil
	\end{equation*}
	iterations, we will reach an iterate $\bfu_k$ satisfying $\norm{\bfu_k - \bfu\uast} \leq \varepsilon$.
\end{corollary}

\begin{proof}
	According to Theorem \ref{thm:complexity}, we choose the stepsize $\tau = T\uast (\varepsilon / 2) / K_{\bfu} (\varepsilon / 2)$ for a sufficiently small $\varepsilon > 0$.
	Then it holds that
	\begin{equation*}
	\begin{aligned}
		\norm{\bfu_{K_{\bfu} (\varepsilon / 2)} - \bfu\uast}
		\leq {} & \norm{\bfu_{K_{\bfu} (\varepsilon / 2)} - \bfu (T\uast (\varepsilon / 2))}
		+ \norm{\bfu (T\uast (\varepsilon / 2)) - \bfu\uast} \\
		\leq {} & \dfrac{\varepsilon}{2} + \dfrac{\varepsilon}{2}
		= \varepsilon,
	\end{aligned}
	\end{equation*}
	which completes the proof.
	\qed
\end{proof}



\section{A Refined Complexity Result}

In this section, we present a more refined complexity result for the gradient descent method \eqref{eq:gd}.
To this end, we assume that the initial point $\bfu_{0}$ satisfies the condition $\norm{\bfu_{0} - \bfu\uast} \leq \delta$.
It is noteworthy that this condition is not particularly stringent.
As shown by the complexity results derived in Section \ref{sec:complexity}, an initial point satisfying this requirement can always be obtained within a finite number of iterations by the gradient descent method \eqref{eq:gd}.

\begin{theorem} \label{thm:refined}
	Suppose that Assumption \ref{asp:function} holds and $\varepsilon > 0$ is a sufficiently small constant.
	Let $\{\bfu_k\}$ be the sequence generated by the gradient descent method \eqref{eq:gd} with $\norm{\bfu_{0} - \bfu\uast} \leq \delta$ and
	\begin{equation*}
		\tau = \mu \beta^{-2} \varepsilon^{2 - 2 \alpha}.
	\end{equation*}
	Then after at most
	\begin{equation*}
		\left\lceil \dfrac{2 \beta^2}{\mu^2} \log \dkh{\dfrac{\norm{\bfu_{0} - \bfu\uast}}{\varepsilon}} \varepsilon^{2 \alpha - 2} \right\rceil
	\end{equation*}
	iterations, we will reach an iterate $\bfu_k$ satisfying $\norm{\bfu_k - \bfu\uast} \leq \varepsilon$.
\end{theorem}

\begin{proof}
	Let $K\uast_\varepsilon$ be the smallest iteration number $k$ such that $\norm{\bfu_k - \bfu\uast} \leq \varepsilon$ holds.
	We first prove that $\norm{\bfu_{k} - \bfu\uast} \leq \delta$ for all $k < K\uast_\varepsilon$ by induction.
	It is obvious that this assertion holds for $k = 0$ by our initialization.
	Now we assume that $\norm{\bfu_{k} - \bfu\uast} \leq \delta$ for some $k < K\uast_\varepsilon$.
	
	By virtue of the strong convexity of $f$, it holds that
	\begin{equation*}
		\jkh{\nabla f (\bfu_{k}), \bfu_{k} - \bfu\uast}
		= \jkh{\nabla f (\bfu_{k}) - \nabla f (\bfu\uast), \bfu_{k} - \bfu\uast}
		\geq \mu \norm{\bfu_{k} - \bfu\uast}^2.
	\end{equation*}
	Since we have $\norm{\bfu_{k} - \bfu\uast} \leq \delta$, the local $\alpha$-H{\"o}lder continuity of $\nabla f$ results in that
	\begin{equation*}
		\norm{\nabla f (\bfu_{k})}^2
		= \norm{\nabla f (\bfu_{k}) - \nabla f (\bfu\uast)}^2
		\leq \beta^2 \norm{\bfu_{k} - \bfu\uast}^{2 \alpha}.
	\end{equation*}
	Hence, it holds that
	\begin{equation*}
	\begin{aligned}
		\norm{\bfu_{k + 1} - \bfu\uast}^2
		= {} & \norm{\bfu_{k} - \tau \nabla f (\bfu_{k}) - \bfu\uast}^2 \\
		= {} & \norm{\bfu_{k} - \bfu\uast}^2
		- 2 \tau \jkh{\nabla f (\bfu_{k}), \bfu_{k} - \bfu\uast}
		+ \tau^2 \norm{\nabla f (\bfu_{k})}^2 \\
		\leq {} & \norm{\bfu_{k} - \bfu\uast}^2
		- 2 \mu \tau \norm{\bfu_{k} - \bfu\uast}^2
		+ \beta^2 \tau^2 \norm{\bfu_{k} - \bfu\uast}^{2 \alpha} \\
		= {} & \dkh{1 - 2 \mu \tau + \beta^2 \tau^2 \norm{\bfu_{k} - \bfu\uast}^{2 \alpha - 2}} \norm{\bfu_{k} - \bfu\uast}^2.
	\end{aligned}
	\end{equation*}
	According to the definition of $K\uast_\varepsilon$, we have $\norm{\bfu_k - \bfu\uast} > \varepsilon$ for any $k < K\uast_\varepsilon$.
	Then a straightforward verification reveals that
	\begin{equation} \label{eq:u-k+1}
	\begin{aligned}
		\norm{\bfu_{k + 1} - \bfu\uast}^2
		\leq {} & \dkh{1 - 2 \mu \tau + \beta^2 \tau^2 \varepsilon^{2 \alpha - 2}} \norm{\bfu_{k} - \bfu\uast}^2 \\
		= {} & \dkh{1 - \mu^2 \beta^{-2} \varepsilon^{2 - 2 \alpha}} \norm{\bfu_{k} - \bfu\uast}^2,
	\end{aligned}
	\end{equation}
	which further implies that
	\begin{equation*}
		\norm{\bfu_{k + 1} - \bfu\uast} \leq \norm{\bfu_{k} - \bfu\uast}
		\leq \delta.
	\end{equation*}
	Consequently, we can conclude that $\norm{\bfu_{k} - \bfu\uast} \leq \delta$ for all $k < K\uast_\varepsilon$.
	
	As a direct consequence of \eqref{eq:u-k+1}, we can proceed to show that
	\begin{equation*}
		\norm{\bfu_{k} - \bfu\uast}^2
		\leq \dkh{1 - \mu^2 \beta^{-2} \varepsilon^{2 - 2 \alpha}}^k \norm{\bfu_{0} - \bfu\uast}^2,
	\end{equation*}
	for all $k < K\uast_\varepsilon$.
	Solving $\dkh{1 - \mu^2 \beta^{-2} \varepsilon^{2 - 2 \alpha}}^k \norm{\bfu_{0} - \bfu\uast}^2 \leq \varepsilon^2$ gives rise to that
	\begin{equation*}
		K\uast_\varepsilon
		\leq \dfrac{2 \log \dkh{\norm{\bfu_{0} - \bfu\uast} / \varepsilon}}{- \log \dkh{1 - \mu^2 \beta^{-2} \varepsilon^{2 - 2 \alpha}}}
		\leq \dfrac{2 \beta^2}{\mu^2} \log \dkh{\dfrac{\norm{\bfu_{0} - \bfu\uast}}{\varepsilon}} \varepsilon^{2 \alpha - 2}.
	\end{equation*}
	The proof is completed.
	\qed
\end{proof}

Theorem \ref{thm:refined} demonstrates that the iteration complexity of the gradient descent method \eqref{eq:gd} is $O (\log (\varepsilon^{-1}) \varepsilon^{2 \alpha - 2})$ with a fixed stepsize for $\alpha \in (0, 1]$.
This complexity result generalizes the classical linear convergence when $\alpha = 1$ and highlights the performance degradation incurred by non-Lipschitz gradients.

\section{Numerical Experiments}

Preliminary numerical results are presented in this section to provide additional insights into the performance guarantees of the gradient descent method \eqref{eq:gd}.
We aim to elucidate that the final error attained by the gradient descent method \eqref{eq:gd} is influenced by both the stepsize $\tau$ and the H{\"o}lder exponent $\alpha$.
All codes are implemented in MATLAB R2018b on a workstation with dual Intel Xeon Gold 6242R CPU processors (at $3.10$ GHz$\times 20 \times 2$) and $510$ GB of RAM under Ubuntu 20.04.

We consider a numerical example from \cite{Barrett1991finite}.
This problem is to solve the following PDE,
\begin{equation} \label{eq:cF}
	\cF (u) = - \nabla^2 u + \nu u_+^{1/2} = 0,
\end{equation}
on the unit square $D = (0, 1) \times (0, 1)$ with the following boundary conditions,
\begin{equation*}
	u = 1 \mbox{~on~} \partial D.
\end{equation*}
Here, $\nu > 0$ is a constant.
It should be noted that $\cF$ is the gradient of the following energy functional,
\begin{equation*}
	\hat{f} (u) = \frac{1}{2} \|\nabla u\|^2 + \frac{2 \nu}{3} \int_D u_+^{3/2} (y) \, \rmd y.
\end{equation*}

Discretizing \eqref{eq:cF} with the standard five point difference scheme \cite{LeVeque2007finite} leads to the following nonlinear system,
\begin{equation*}
	\bfF (\bfu) = \bfL \bfu + \nu \bfu_+^{1/2} - \bfb = 0,
\end{equation*}
where $\bfL$ is the discretization of $- \nabla^2$ with zero boundary conditions, $\bfb$ encodes the boundary conditions, and $\bfu_+^{1/2}$ is understood as a component-wise operation.
Then the discrete problem has the same properties as the continuous problem.
It is noteworthy that the matrix $\bfL$ is symmetric positive definite.
Moreover, we have $\bfF = \nabla \bar{f}$ and $\bar{f}$ is the discretization of $\hat{f}$ as follows,
\begin{equation*}
	\bar{f} (\bfu) = \dfrac{1}{2} \bfu\zz \bfL \bfu + \frac{2\nu}{3} \bfe\zz \bfu_+^{3/2} - \bfb\zz \bfu,
\end{equation*}
where $\bfe$ is the vector of all ones.



To evaluate the performance of the gradient descent method \eqref{eq:gd}, we focus on the following optimization problem inspired by the above PDE model,
\begin{equation} \label{opt:test}
	\min_{\bfu \in \Rn} f (\bfu) = \dfrac{1}{2} \bfu\zz \bfA \bfu + \dfrac{1}{1 + \alpha} \bfe\zz \bfu_{+}^{1 + \alpha} - \bfc\zz \bfu,
\end{equation}
where $\bfA \in \Rnn$ is a symmetric positive definite matrix, $\alpha \in (0, 1)$ is a constant, and $\bfc = \bfA \bfu\uast + (\bfu\uast)_{+}^{\alpha} \in \Rn$ is a vector with $\bfu\uast \in \Rn$.
It is evident that the objective function $f$ is strongly convex with $\mu=\lambda_{\min}(\bfA)$ and its gradient $\nabla f$ is locally $\alpha$-H{\"o}lder continuous with $\beta = 1 +\lambda_{\max}(\bfA)$ and $\delta= 1$.
Moreover, a straightforward verification reveals that $\bfu\uast$ is the minimizer of problem \eqref{opt:test}.

In our numerical experiments, the initial point $\bfu_{0}$, the minimizer $\bfu\uast$ and the matrix $\bfA$ in the test problem \eqref{opt:test} are generated randomly, and the vector $\bfc$ is defined by $\bfu\uast, \bfA$ and $\alpha$ with the detailed MATLAB code provided as follows.
\begin{lstlisting}
	u_0 = randn(n, 1);
	u_star = randn(n, 1);
	A = randn(n); A = A'*A + eye(n);
	c = A*u_star + max(u_star, 0).^alpha;
\end{lstlisting}
Moreover, the test problem dimension is fixed at $n = 50$, and method \eqref{eq:gd} is permitted a maximum of $10000$ iterations.

In the first experiment, we scrutinize the performance of the gradient descent method \eqref{eq:gd} under different stepsizes.
Specifically, with the parameter $\alpha$ fixed at $0.5$, the algorithm is tested for stepsizes chosen from the set $\{0.01, 0.005, 0.001, 0.0005\}$.
The corresponding numerical results, presented in Figure \ref{subfig:stepsize}, illustrate the decay of the distance between the iterates and the global minimizer over iterations.
It can be observed that a larger stepsize facilitates a more rapid descent  in the early stage of iterations, albeit at the expense of a greater asymptotic error.
This phenomenon corroborates our theoretical predictions.

In the second experiment, the stepsize $\tau$ is fixed at $0.001$, while the parameter $\alpha$ is varied over the values $\{0.2, 0.4, 0.6, 0.8\}$.
Figure \ref{subfig:alpha} similarly tracks the decay of the distance to the global minimizer over iterations.
It is evident that, as the value of $\alpha$ decreases, the final error attained by the algorithm increases under the same stepsize.
Therefore, the associated optimization problems become increasingly ill-conditioned and thus more challenging to solve for smaller values of $\alpha$.
These findings offer empirical support for our theoretical analysis.

\begin{figure}[t]
	\centering
	\subfigure[different stepsizes]{
		\label{subfig:stepsize}
		\includegraphics[width=0.476\linewidth]{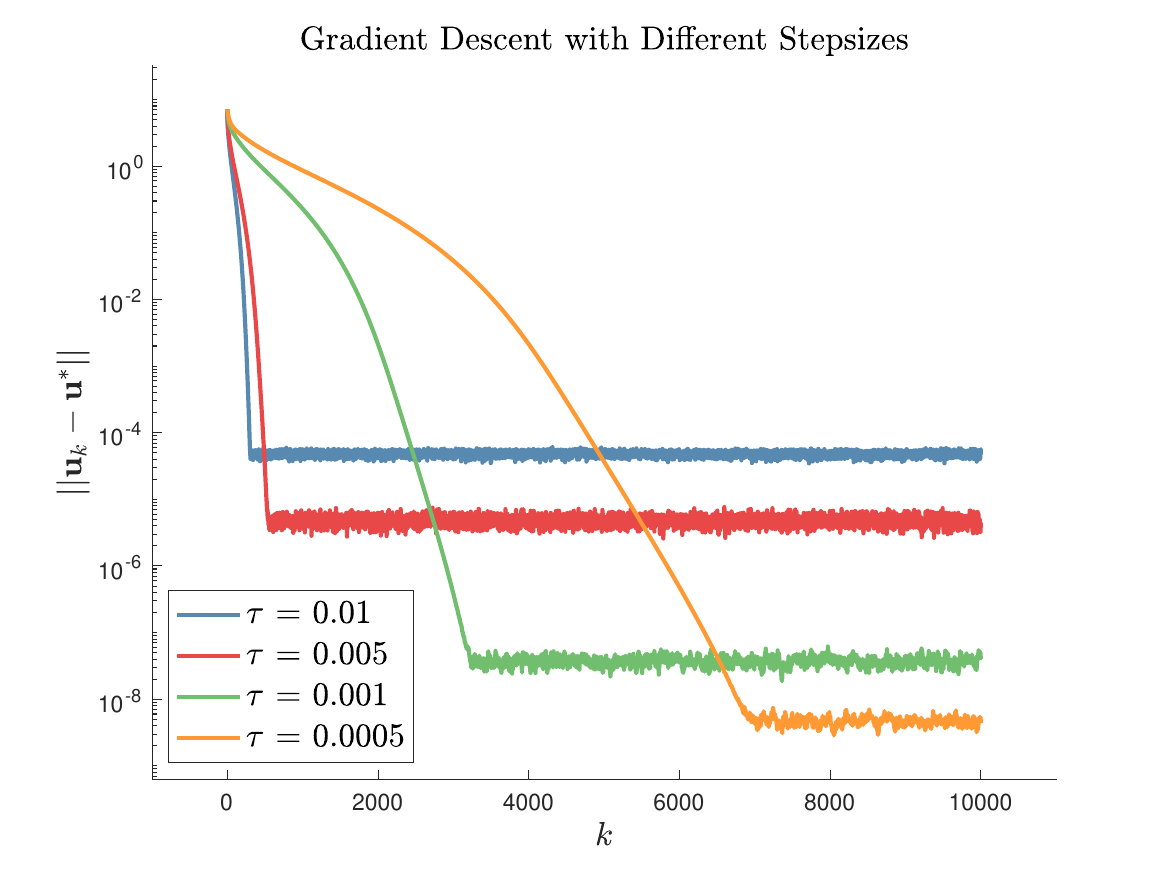}
	}
	\subfigure[different values of $\alpha$]{
		\label{subfig:alpha}
		\includegraphics[width=0.476\linewidth]{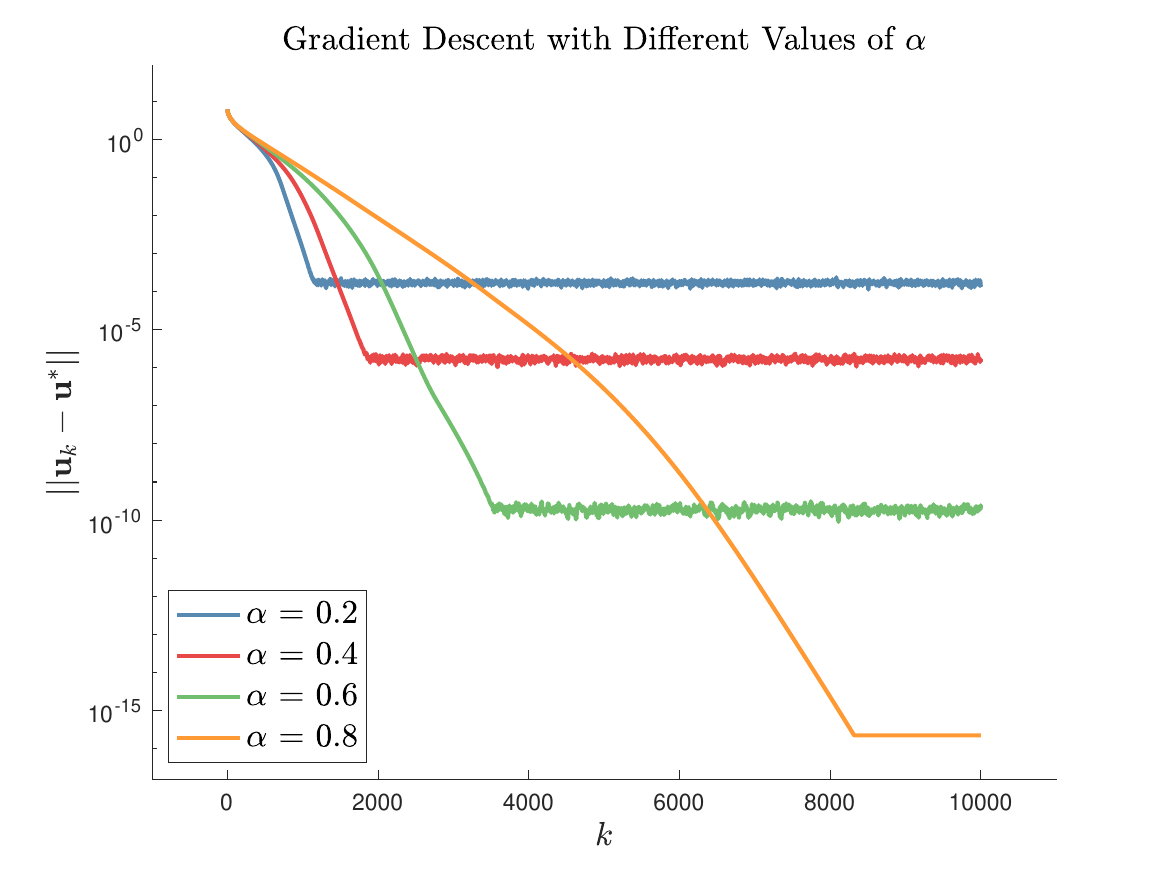}
	}
	\caption{Numerical performance of gradient descent method \eqref{eq:gd} for problem \eqref{opt:test}.}
	\label{fig:gd}
\end{figure}

\section{Conclusion}

In this paper, we have established a new complexity result for the gradient descent method with a fixed stepsize applied to strongly convex optimization problems where the objective function possesses a locally $\alpha$-H{\"o}lder continuous gradient with $\alpha \in (0, 1]$.
Our analysis interprets the gradient descent method as an explicit Euler discretization of a gradient flow, and provides rigorous estimates on the discretization error.
Based on these results, we have derived both a basic complexity result without any additional assumptions and a refined version under a mild condition on the initial point.
The main conclusion is that, to achieve an approximate minimizer with a distance to the minimizer less than $\varepsilon$, the total number of iterations required by the gradient descent method \eqref{eq:gd} is $O (\log (\varepsilon^{-1}) \varepsilon^{2 \alpha - 2})$ at the most.
This recovers the classical complexity result when $\alpha = 1$ and reveals the additional difficulty imposed by the weaker smoothness of the objective function for $\alpha < 1$.
Numerical experiments are conducted to validate our theoretical findings, demonstrating the expected behavior of gradient descent under different stepsizes and H{\"o}lder exponents.
These results offer new insights into the performance guarantees of the classic gradient descent method for a broader class of optimization problems with non-Lipschitz gradients.


%
%

\section*{Statements and Declarations}

\paragraph{Competing Interests}
The authors have no competing interests to declare that are relevant to the content of this paper.

\paragraph{Data Availability}
The authors declare that all data supporting the findings of this study are available within this paper.


\bibliographystyle{spmpsci}
\bibliography{library_HGD}

\begin{thebibliography}{1}
\providecommand{\url}[1]{{#1}}
\providecommand{\urlprefix}{URL }
\expandafter\ifx\csname urlstyle\endcsname\relax
  \providecommand{\doi}[1]{DOI~\discretionary{}{}{}#1}\else
  \providecommand{\doi}{DOI~\discretionary{}{}{}\begingroup
  \urlstyle{rm}\Url}\fi

\bibitem{Barrett1991finite}
Barrett, J.W., Shanahan, R.M.: Finite element approximation of a model
  reaction--diffusion problem with a non-lipschitz nonlinearity.
\newblock Numer. Math. \textbf{59}, 217--242 (1991)

\bibitem{Devolder2014first}
Devolder, O., Glineur, F., Nesterov, Y.: First-order methods of smooth convex
  optimization with inexact oracle.
\newblock Math. Program. \textbf{146}, 37--75 (2014)

\bibitem{Gronwall1919note}
Gronwall, T.H.: Note on the derivatives with respect to a parameter of the
  solutions of a system of differential equations.
\newblock Ann. Math. \textbf{20}, 292--296 (1919)

\bibitem{Lan2015bundle}
Lan, G.: Bundle-level type methods uniformly optimal for smooth and nonsmooth
  convex optimization.
\newblock Math. Program. \textbf{149}, 1--45 (2015)

\bibitem{LeVeque2007finite}
LeVeque, R.J.: Finite Difference Methods for Ordinary and Partial Differential
  Equations: Steady-State and Time-Dependent Problems.
\newblock Society for Industrial and Applied Mathematics (2007)

\bibitem{Nesterov2015universal}
Nesterov, Y.: Universal gradient methods for convex optimization problems.
\newblock Math. Program. \textbf{152}, 381--404 (2015)

\bibitem{Nesterov2018lectures}
Nesterov, Y.: Lectures on Convex Optimization.
\newblock Springer (2018)

\bibitem{Teschl2012ordinary}
Teschl, G.: Ordinary Differential Equations and Dynamical Systems.
\newblock American Mathematical Society (2012)

\end{thebibliography}

\end{document}